\documentclass[11pt]{amsart}
\usepackage{amsfonts}
\usepackage{graphicx}
\usepackage{amscd}
\usepackage{amsmath}
\usepackage{amssymb}
\usepackage{amsfonts}
\newtheorem{theorem}{Theorem}
\theoremstyle{plain}

\newtheorem{lemma}{Lemma}

\theoremstyle{definition}

\begin{document}
\title[Chebyshev Upper Estimates for Generalized Primes]{Chebyshev Upper Estimates for Beurling's Generalized Prime Numbers}
\author[J. Vindas]{Jasson Vindas}
\address{Department of Mathematics, Ghent University, Krijgslaan 281 Gebouw S22, B 9000 Gent, Belgium}
\email{jvindas@cage.Ugent.be}
\thanks{The author gratefully acknowledges support by a Postdoctoral Fellowship of the Research
Foundation--Flanders (FWO, Belgium)}


\begin{abstract}
Let $N$ be the counting function of a Beurling generalized number system  and let $\pi$ be the counting function of its primes. We show that the $L^{1}$-condition 
$$
\int_{1}^{\infty}\left|\frac{N(x)-ax}{x}\right|\frac{\mathrm{d}x}{x}<\infty
$$
and the asymptotic behavior
$$N(x)=ax+O\left(\frac{x}{\log x}\right)\: ,$$
for some $a>0$, suffice for a Chebyshev upper estimate
$$
\frac{\pi(x)\log x}{x}\leq B<\infty\: .
$$
\end{abstract}
\maketitle
\section{Introduction}
 
Let $P=\left\{p_k\right\}_{k=1}^{\infty}$ be a set of Beurling generalized primes, namely, a non-decreasing sequence of real numbers $
1<p_1\leq p_{2}\leq \cdots\leq p_{k} \to\infty.$
The sequence $\left\{n_{k}\right\}_{k=1}^{\infty}$ denotes its associated set of generalized integers \cite{bateman-diamond,beurling}. Consider the counting functions of generalized integers and primes
\begin{equation*}
N(x)=N_{P}(x)=\sum_{n_{k}<x}1 \ \ \mbox{ and }\ \ 
\pi(x)=\pi_{P}(x)=\sum_{p_{k}<x}1\: .
\end{equation*}
Beurling's problem consists in finding mild conditions over $N$ that ensure a certain asymptotic behavior for $\pi$. This problem has been extensively investigated in connection with the prime number theorem (PNT), i.e., 
\begin{equation}
\label{cueeq1}
\pi(x)\sim \frac{x}{\log x}\:, \ \ \ x\to\infty\: ,
\end{equation}
and Chebyshev two-sided estimates, that is,
\begin{equation}
\label{cueeq2}
0<\liminf_{x\to\infty} \frac{\pi(x)\log x}{x} \ \ \ \mbox{and} \ \ \ \limsup_{x\to\infty} \frac{\pi(x)\log x}{x}<\infty\:.
\end{equation}
On the other hand, there are no mild hypotheses in the literature for Chebyshev upper estimates,
\begin{equation}
\label{cueeq3}
\limsup_{x\to\infty} \frac{\pi(x)\log x}{x}<\infty\:.
\end{equation}
The purpose of this article is to study asymptotic requirements over $N$ that imply the Chebyshev upper estimate (\ref{cueeq3}).

Beurling \cite{beurling} proved that 
\begin{equation}
\label{cueeq4}
N(x)=ax+O\left(\frac{x}{\log^{\gamma}x}\right)\: , \ \ \ x\to\infty\ \ \  (a>0)\: ,
\end{equation}
where $\gamma>3/2$, suffices for the PNT (\ref{cueeq1}) to hold. See \cite{beurling,kahane1,vindasGPNT} for more general PNT. 
Beurling's condition is sharp, because when 
$\gamma=3/2$ there are generalized number systems for which the PNT fails \cite{beurling,diamond2}. For $\gamma<1$, not even Chebyshev estimates need to hold, as follows from an example of Hall \cite{hall} (see also \cite{balanzario2000}). Diamond has shown \cite{diamond3} that (\ref{cueeq4}) with $\gamma>1$ is enough to obtain Chebyshev two-sided estimates (\ref{cueeq2}). Furthermore, he 
conjectured \cite{diamond4} that the weaker hypothesis
\begin{equation}
\label{cueeq5}
\int_{1}^{\infty}\left|\frac{N(x)-ax}{x}\right|\frac{\mathrm{d}x}{x}<\infty\: ,  \ \ \ \mbox{with } a>0\: ,
\end{equation}
would be enough for (\ref{cueeq2}). His conjecture was shown to be false by Kahane \cite{kahane2}. Nevertheless, the author has recently shown \cite{vindaschebyshevI} that if one adds to (\ref{cueeq5}) the condition 
\begin{equation}
\label{cueeq6}
N(x)=ax+o\left(\frac{x}{\log x}\right)\: , \ \ \ x\to\infty\: ,
\end{equation}
then (\ref{cueeq2}) is fulfilled, 
extending thus earlier results from \cite{diamond3,zhang}.

It is natural to replace the little $o$ symbol in (\ref{cueeq6}) by an $O$ growth estimate and investigate the effect of this new condition on the asymptotic distribution of the generalized primes. It turns out that one gets a Chebyshev upper estimate in this case. Our main goal is to give a proof of the following theorem.

\begin{theorem}
\label{cueth1} Diamond's $L^{1}$-condition $(\ref{cueeq5})$ and the asymptotic behavior
\begin{equation}
\label{cueeq7}
N(x)=ax+O\left(\frac{x}{\log x}\right)\:, \ \ \ x\to\infty\:,
\end{equation}
suffice for the Chebyshev upper estimate $(\ref{cueeq3})$.
\end{theorem}

\section{Notation}
We will give an analytic proof of Theorem \ref{cueth1}. Our technique follows distributional ideas already used in \cite{vindasGPNT,vindaschebyshevI,vindas-estradaPNT}. It employs the Wiener division theorem \cite[Chap. 2]{korevaar} and the operational calculus for the Laplace transform of Schwartz distributions \cite{bremermann,vladimirov}. The Schwartz spaces of test functions and distributions are denoted as $\mathcal{D}(\mathbb{R})$, $\mathcal{S}(\mathbb{R})$, $\mathcal{D}'(\mathbb{R})$ and $\mathcal{S}'(\mathbb{R})$, see \cite{estrada-kanwal,schwartz,vladimirov} for their properties. If $f\in\mathcal{S}'(\mathbb{R})$ has support in $[0,\infty)$, its Laplace transform is well defined as
$$
\mathcal{L}\left\{f;s\right\}=\left\langle f(u),e^{-su}\right\rangle\: , \ \ \  \Re e\:s>0\: ,
$$
and the Fourier transform $\hat{f}$ is the distributional boundary value \cite{bremermann} of $\mathcal{L}\left\{f;s\right\}$ on $\Re e\:s=0$. We use the notation $H$ for the Heaviside function, it is simply the characteristic function of $(0,\infty)$.

Observe that (\ref{cueeq3}) is equivalent to 
\begin{equation}
\label{cueeq8}
\limsup_{x\to\infty}\frac{\psi(x)}{x}<\infty\:,
\end{equation}
where $\psi$ is the Chebyshev function
$$
\psi(x)=\psi_{P}(x)=\sum_{n_{k}<x}\Lambda(n_{k})\:,
$$
as follows from \cite[Lem. 2E]{bateman-diamond}.

\section{Proof of Theorem \ref{cueth1}}
Assume (\ref{cueeq5}) and (\ref{cueeq7}). Set $T(u)=e^{-u}\psi(e^{u})$. We must show (\ref{cueeq8}), that is,
\begin{equation}
\label{cueeq9}
\limsup_{u\to\infty}T(u)<\infty\: . 
\end{equation}
The crude inequality $T(u)\leq ue^{-u}N(e^{u})=O(u)$ implies that $T\in\mathcal{S}'(\mathbb{R})$. The proof of (\ref{cueeq9}) depends upon estimates on convolution averages of $T$:

\begin{lemma}
\label{cuel1} There exists $c>0$ such that
\begin{equation}
\label{cueeq10}
\int_{-\infty}^{\infty}T(u)\hat{\phi}(u-h)\mathrm{d}u=O(1)\:,
\end{equation}
whenever $\phi\in \mathcal{D}(-c,c)$.
\end{lemma} 
Indeed, suppose that Lemma \ref{cuel1} has been already established. Choose then in (\ref{cueeq10}) a test function $\phi\in\mathcal{D}(-c,c)$ such that $\hat{\phi}$ is non-negative.
Since $\psi(e^{u})$ is non-decreasing, we have $e^{-u}T(h)\leq T(u+h)$ whenever $u$ and $h$ are positive. Setting $C=\int_{0}^{\infty}e^{-u}\hat{\phi}(u)\mathrm{d}u>0,$ we obtain that
\begin{equation*}T(h)\leq \frac{1}{C}\int_{0}^{\infty}T(u+h)\hat{\phi}(u)\mathrm{d}u=O(1)\:,
\end{equation*}
and Theorem \ref{cueth1} follows at once. It remains to prove the lemma.

\begin{proof}[Proof of Lemma \ref{cuel1}] Set $E_{1}(u):=e^{-u}N(e^{u})-aH(u)$ and $E_{2}(u)=uE_{1}(u)$. The assumptions (\ref{cueeq5}) and (\ref{cueeq7}) take the form $E_{1}\in L^{1}(\mathbb{R})$ and $E_{2}\in L^{\infty}(\mathbb{R})$. Consider
$$
G(s)=\zeta(s)-\frac{a}{s-1}= s\mathcal{L}\left\{E_{1};s-1\right\}+a \: . 
$$
Taking $\Re e\:s\to1^{+}$, in the distributional sense, we obtain $G(1+it)=(1+it)\hat{E}_{1}(t)+a$. Since $E_1\in L^{1}(\mathbb{R})$, $\hat{E}_{1}$ is continuous; therefore $G(s)$ extends to a continuous function on $\Re e\:s=1$. Consequently, $(s-1)\zeta(s)$ is continuous on $\Re e\:s=1$ and there exists $c>0$ such that $it\zeta(1+it)\neq 0$ for all $t\in(-3c,3c)$. Next, we study the boundary values, on the line segment $1+i(-c,c)$, of 
$$
\mathcal{L}\left\{T(u);s-1\right\}=\mathcal{L}\left\{\psi(e^{u});s\right\}=-\frac{\zeta'(s)}{s\zeta(s)}\: .
$$
A quick calculation shows that
\begin{equation}
\label{cueeq11}
-\frac{\zeta'(s)}{s\zeta(s)}=\frac{\mathcal{L}\left\{E'_{2};s-1\right\}}{(s-1)\zeta(s)}-\frac{(2s-1)\mathcal{L}\left\{E_1;s-1\right\}+a}{s(s-1)\zeta(s)}-\frac{1}{s}+\frac{1}{s-1}\: ,
\end{equation}
Consider the boundary distributions
$$
g_{1}(t)=\lim _{\sigma\to1^{+}}\frac{\mathcal{L}\left\{E'_{2};\sigma-1+it\right\}}{(\sigma-1+it)\zeta(\sigma+it)} \ \ \ \mbox{in }\mathcal{S}'(\mathbb{R})\: , 
$$ 
and 
$$
g_{2}(t)=-\lim _{\sigma\to1^{+}}\left(\frac{(2\sigma-1+2it)\mathcal{L}\left\{E_{1};\sigma-1+it\right\}+a}{(\sigma+it)(\sigma-1+it)\zeta(\sigma+it)}+\frac{1}{\sigma+it}\right) \ \ \ \mbox{in }\mathcal{S}'(\mathbb{R})\: .
$$
Taking boundary values in (\ref{cueeq11}),
we have
$
\hat{T}(t)=g_{1}(t)+ g_{2}(t)+\hat{H}(t),
$
where $H$ is the Heaviside function. Fix $\phi\in\mathcal{D}(-c,c)$. Notice that $g_{2}$ is actually a continuous function on $(-3c,3c)$, thus,
\begin{align*}
\int_{-\infty}^{\infty}T(u)\hat{\phi}(u-h)\mathrm{d}u&=\left\langle g_{1}(t),e^{iht}\phi(t)\right\rangle+\int_{-c}^{c}e^{iht}g_{2}(t)\phi(t)\mathrm{d}t +\int_{-h}^{\infty}\hat{\phi}(u)\mathrm{d}u
\\
&
=\left\langle g_{1}(t),e^{iht}\phi(t)\right\rangle+o(1)+O(1)\: .
\end{align*}
Our task is then to demonstrate that
$
\left\langle g_{1}(t),e^{iht}\phi(t)\right\rangle=O(1).
$
Let $M\in\mathcal{S}'(\mathbb{R})$ be the distribution supported in the interval $[0,\infty)$ that satisfies $\mathcal{L}\left\{M;s-1\right\}=((s-1)\zeta(s))^{-1}$. Notice that $g_{1}=\widehat{( E_{2}'\ast M)}$. Fix an even function $\eta\in \mathcal{D}(-3c,3c)$ such that $\eta(t)=1$ for all $t\in(-2c,2c)$. Then, $\eta(t)it\zeta(1+it)\neq0$ for all $t\in(-2c,2c)$; moreover, it is the Fourier transform of the $L^{1}$-function $\chi_{1}\ast E_{1}+\chi_{2}$, where $\hat{\chi}_{1}(t)=it(1+it)\eta(t)$ and $\hat{\chi}_{2}(t)=a(1+it)\eta(t)$. We can therefore apply the Wiener division theorem \cite[p. 88]{korevaar} to $\eta(t)it\zeta(1+it)$ and $\phi(t)$. So we find $f\in L^{1}(\mathbb{R})$ such that

$$
\hat{f}(t)=\frac{\phi(t)}{\eta(t) it \zeta(1+it)}\: .
$$
Hence,
$$\left\langle g_{1}(t),e^{iht}\phi(t)\right\rangle=\left\langle (E_{2}'\ast M)(u),\hat{\phi}(u-h)\right\rangle=(E_{2}\ast (\hat{\eta})'\ast f)(h)=O(1)\: , $$
because $E_{2}\in L^{\infty}(\mathbb{R})$ and $(\hat{\eta})'\ast f\in L^{1}(\mathbb{R})$, whence (\ref{cueeq10}) follows. 

\end{proof}

\end{document}